\documentclass[11pt,leqno]{article}
\usepackage{amsmath,amssymb,amsfonts,graphicx,}
\newtheorem{theorem}{Theorem}[section]
\newtheorem{proposition}[theorem]{Proposition}

\newtheorem{lemma}[theorem]{Lemma}

\newtheorem{definition}[theorem]{Definition}
\newtheorem{remark}[theorem]{Remark}

\newtheorem{example}[theorem]{Example}

\def\d{\hbox{\bf d}{} }

\def\Q{\hbox{\bf $\mathbb{Q}$}{} }
\def\T{\hbox{\bf T}{} }

\def\OCP{\hbox{\bf OCP}{} }
\def\WCN{\hbox{\bf WCN}{} }

\def\MP{\hbox{\bf MP}{} }

\def\KS{\hbox{\bf KS}{} }

\newcommand{\ra }{\rightarrow }

\newenvironment{proof}{\mbox{\bf Proof.}}{\mbox{$\dashv$}\bigskip}

\begin{document}

\begin{center}
{\Large\bf The Temporal Continuum}\\
\vspace{.25in}
{\bf Mohammad Ardeshir~\footnote{
Sharif University of Technology, 11365-9415, Tehran, Iran, mardeshir@sharif.edu} ~~~~~ Rasoul Ramezanian~\footnote{University of Lausanne, Internef, 1015 Lausanne,
Switzerland, rasoul.ramezanian@unil.ch} 
 } 
\end{center}
\begin{abstract}
\noindent  The continuum has been one of the most controversial
topics in mathematics since the time of the Greeks. Some mathematicians, such as
Euclid and Cantor, held the position that a line is composed of
points, while others, like Aristotle, Weyl and Brouwer, argued that a line is not composed of points but rather a matrix of
a continued insertion of points. In spite of this disagreement on
the structure of the continuum, they did distinguish the
{\em temporal line} from the {\em spatial line}.  In this paper, we  argue
that  there is indeed a difference between the intuition of the
spatial continuum and the intuition of the temporal continuum. The
main primary aspect of the  temporal continuum, in
contrast with the spatial continuum, is the notion of  {\em orientation}.\\

The continuum has usually been mathematically modeled by 
\emph{Cauchy} sequences and the \emph{Dedekind} cuts. While in the
first model, each point can be approximated by rational numbers,
in the second one, that is not  possible constructively. We argue
that points on the temporal continuum cannot be approximated by
rationals as a temporal point is a \emph{flow} that sinks to the past. In our model, the continuum is a collection of constructive
\emph{Dedekind} cuts, and we define two topologies for 
temporal continuum: 1. \emph{oriented} topology and 2. the
\emph{ordinary} topology. We prove that every total function from
the \emph{oriented}
topological space to the \emph{ordinary} one is continuous.   \\
\vspace*{1 cm}
\ \\
\end{abstract}

\section{Introduction}
\label{Int}

 Mathematicians have long been divided into two philosophical camps regarding the structure of the continuum. Some, like Euclid \cite{kn:Eud} and Cantor \cite{kn:ca}, view the continuum as a composition of points. On the other hand, others, including Aristotle \cite{kn:Ars}, Weyl \cite{kn:BW,kn:weyl}, and Brouwer \cite{kn:Br}, consider it as a \emph{whole}, not composed of points.

In this paper, we aim to highlight another distinction, namely, between what we refer to as \emph{the spatial continuum} and \emph{the temporal continuum}. We will model the temporal continuum using a specific type of Dedekind cuts that illustrate this difference. 

 The main distinction between the spatial continuum and the temporal continuum lies in the notion of \emph{orientation}: the temporal continuum is oriented, moving exclusively from the past to the future, and it is impossible to move in the \emph{opposite} direction\footnote{In this paper, we \emph{assume} that `time is directed' without delving into detailed arguments. Numerous works have explored the direction of time from both physical and philosophical perspectives; for further reading, refer to \cite{kn:im, kn:Rei}.}. Our subjective experience reinforces this clear differentiation between the future and the past. We can vividly remember the past, whereas the future remains uncertain. Our ignorance of future events, including our own choices and actions, contributes to the concept of \emph{free will}. Conversely, we lack the ability to alter the past due to the absence of free will in that direction.

\begin{quote}{\footnotesize We can affect the future by our actions: so why can we not
by our actions affect the past? The answer that springs to mind is
this: you cannot change the past; if a thing has happened, it has
happened, and you cannot make it not to have
happened~\cite{kn:DU}.}
\end{quote}
 The primary objective of this paper is to propose a model for the continuum that incorporates the concept of orientation. To achieve this, we will introduce a formalization of orientation through a specific type of Dedekind cuts, which we will refer to as oriented Dedekind cuts.

 Brouwer and Weyl emphasized two essential aspects of the intuitive continuum, which are \emph{inexhaustibility} and \emph{non-discreteness} (also referred to as \emph{non-atomicity}) (\cite{kn:va}, page 86). The concept of choice sequences is motivated by these characteristics. Non-lawlike sequences signify the continuum's inexhaustibility, and identifying points with unfinished sequences of nested intervals reflects its non-discreteness (\cite{kn:va}, page 87).  Defining points as choice sequences of nested intervals captures the notion that a point on the continuum is not a dimension-less atom but rather a \emph{halo}.

As choice sequences are non-predetermined and unfinished entities, the value of every total function from choice sequences to natural numbers relies on an initial segment of sequences. Brouwer leverages this property to demonstrate that every total function over the continuum is continuous.

 Both the spatial continuum and the temporal continuum share the aforementioned aspects. However, a defining characteristic of the temporal continuum is the concept of \emph{orientation}. "Duration" is a continuous \emph{becoming}, always moving towards the future. In the spatial continuum, movement is possible in both directions. To introduce witnesses for points on the spatial line, one may proceed by introducing witnesses for all points in the segment $[-1,0]$, then for points in the segment $[1,2]$, and subsequently for points in the segment $[0,1]$, and so on. There is no requirement to adhere to any specific direction.
However, for the temporal line, this approach is not feasible. Let the locations of a moving ball in space serve as witnesses for moments. If we are at the moment $t_0$, and $t_1$ is a future moment, we cannot determine the location of the ball at time $t_1$ at the present moment. We must wait until we reach that moment, and only then will the locations of the ball in all moments before $t_1$ be determined. \emph{We are constantly moving towards the future}. The following \emph{assertion} may provide further clarity regarding our understanding of the notion of orientation. We leverage notion of orientation to demonstrate that every total function over the temporal continuum is continuous.


 In considering how we experience a point on the temporal continuum, we can conceptualize a point $t$ as the moment of occurrence of an event $E$:
\begin{center}
    t:="The moment of the occurrence of the event E"
\end{center}
For instance, assume it is currently 8:00 am, and I am aware that the event $E$ will occur sometime after 8:00 am but before 12:00 pm. As time passes, I experience the point $t$ in the following manner: I examine the occurrence of $E$ at several consecutive times. To do this, I plan a strategy and choose specific moments to observe the event. The choice of these moments is up to me.

For example, I may decide to examine the occurrence of $E$ at times like 10:15 am, 10:20 am, 10:40 am, 11:30 am, and so on. Suppose it is now 11:30 am, and the last time I observed the occurrence of $E$ was at 10:40 am, where I found no evidence of it happening yet. Now, suppose it is 11:30 am again, and I observe the event and find evidence of its occurrence. However, at this point, I \emph{cannot move to the past} to examine its occurrence at any previous moments.

All the information I have about $t$ is that it lies \emph{sometime} within the interval $(10:40, 11:30]$. I cannot distinguish the moment $t$ from other moments within this interval. In other words, \emph{the point $t$ has "sunk back" to the past, and I cannot experience it any further}. It is \emph{absolutely undecidable} whether $t$ is before 11:00 am or after it. The moment $t$ cannot be estimated more accurately, and consequently, I cannot approximate $t$ using rational moments.

Furthermore, if I have a constructive method to introduce witnesses for all moments on the temporal line, then the witness for the moment $t$ is also the witness for all moments within the interval $(10:40, 11:30]$. Thus, a point on the temporal continuum is like a \emph{flow} that sinks back to the past eventually.

We consider the above \emph{explanation} as our understanding of the notion of \emph{orientation} and attempt to formalize it in this paper. In our perspective, we regard $t$ as a \emph{flow} from the past, meaning it encompasses all moments that have occurred before $t$. If $t$ is a past moment, we cannot obtain any additional information about which moments belong to this collection and which ones do not. If $t$ is a future moment, we still have time to gather data about the elements of the collection before $t$ eventually sinks back to the past. During this time, we may examine the membership status of certain moments. However, as soon as $t$ sinks back to the past, we can no longer acquire any new data.

This passage discusses the experience of a point on the temporal continuum, emphasizing the concept of \emph{orientation} and how it is related to the passage of time and our ability to observe and distinguish moments. It also introduces the idea of a point being a \emph{flow} from the past, capturing the continuous and irreversible nature of time.


 Brouwer's perspective on the continuum is that it is intuitively given as a flowing medium of cohesion between two events, not comprised of points (events) itself, but rather an inexhaustible matrix allowing for a continued insertion of points. Originally, there are no points on the continuum, but we can construct points on it or indicate a position within it. Brouwer emphasized that the intuition of the continuum is the intuition of the medium of cohesion between two events. He distinguishes between two things: the medium of cohesion (first thing) and the continuum (second thing), or \emph{primum} and \emph{secundum} as he puts it (\cite{kn:jk}, page 70). Brouwer utilized the category of choice sequences and the continuity principle to provide a mathematical analysis of continua. Using choice sequences to define points on the continuum, a point is modeled as a \emph{halo}. Consequently, to construct a witness for a point on the continuum, the witness is constructed for a halo around it. Brouwer proposed his famous continuity principle for choice sequences and used it to prove that every total function over the real line is continuous (\cite{kn:TD}, page 305).

Motivated by the characteristic aspect of the temporal continuum, i.e., the \emph{orientation}, as explained above, we introduce oriented cuts to model points on the temporal continuum as \emph{flows}. The traditional modeling of the continuum using the \emph{Cauchy} fundamental sequences of rationals \cite{kn:TD} allows every real number to be approximated by rational numbers. However, since moments on the temporal continuum sink back to the past, they \emph{cannot} be approximated by rationals. Consequently, the \emph{Cauchy} sequences appear unsuitable for modeling the temporal continuum. Instead, Dedekind cuts prove to be more appropriate for this purpose. Among the constructive Dedekind lines introduced in \cite{kn:TD}, only $\mathbb{R}^d$ can be positively approximated by rationals due to its \emph{locatedness} property (\cite{kn:TD}, page 270), while $\mathbb{R}^e$ and $\mathbb{R}^{be}$ can be approximated by rationals only up to the double negation via the \emph{strong monotonicity} property. We demonstrate that the collection of oriented Dedekind cuts cannot be approximated by rationals, as desired in this case. As discussed earlier, points on the temporal continuum are like flows that sink back to the past. Once they have sunk back, we cannot acquire new data about them, and thus, they cannot be approximated by rationals.


 In addition to our main objective of introducing a mathematical model for the temporal continuum, we also aim to demonstrate that every total function over the temporal line is continuous, similar to the Brouwerian real line. For this purpose, we require a \emph{continuity principle} for the Dedekind cuts. Hence, we propose a principle called the \emph{oriented continuity principle} (\OCP) and define two topologies for the temporal continuum: 1. the \emph{oriented topology} and 2. the \emph{ordinary topology}. Utilizing \OCP, we establish that every total function from the oriented topological space to the ordinary one is continuous.

The
sequel of the paper is organized as follows. In
section~\ref{orientedcut}, we define the notion of the oriented
cuts, and   justify a continuity principle for them. The
oriented continuity principle, $\OCP$, expresses formally the
feature we have in mind about a continuity principle.  We show
that the oriented reals cannot be approximated by rationals. in
section~\ref{topforc}, the \emph{oriented topology} and the
\emph{ordinary topology} are defined, and then some consequences
of \OCP for the temporal continuum are demonstrated.

We argue in the context of constructive logic. As far as possible,
the standard notations in \cite{kn:TD} are used in this paper.


\section{Oriented Cuts}\label{orientedcut}

In this section, we introduce a type of left cuts of $\mathbb{Q}$,
named \emph{oriented cuts},  and  state the oriented
continuity principle.

\begin{definition}\label{o-cut}
We let $\mathbb{R}^o$ be the set of all strictly increasing
bounded sequences  of rational numbers, i.e.
$\alpha\in\mathbb{R}^o$ if and only if $\exists M\forall
n(\alpha(n)<\alpha(n+1)<M)$. For all $\alpha,\beta\in
\mathbb{R}^o$ we define:
\begin{itemize}
\item $\alpha\leq \beta$
if and only if  $\forall m\exists n (\alpha(m)<\beta(n))$,
\item  $\alpha< \beta$
if and only if $\exists n\forall m (\alpha(m)<\beta(n))$, and
\item
$\alpha=^o\beta$ if and only if  $\alpha\leq \beta\wedge\beta\leq
\alpha$.

\end{itemize} We call the set $\mathbb{R}^o$,
regarding equality $=^o$, the set of all oriented reals
\emph{(}cuts\emph{)}.
\end{definition}

  Oriented reals, bounded  strictly increasing sequences of
rationals, are supposed to model the passing of  time. The
increase of sequences reflects the  passing of time and the strictness of
increase ensures that time cannot rest.

\begin{definition} For a rational number $q\in \Q$,  let
$\hat{q}$ be the oriented real defined by
$\hat{q}(n)=q-\frac{1}{n+1}$, for all $n$.  Let
$q\rightarrow \hat{q}$ be the mapping which assigns the oriented
real $\hat{q}$ to $q$. For each $q\in \mathbb{Q}$ and
$\alpha\in\mathbb{R}^o$, we say

\begin{itemize}
\item[a.] $q<\alpha$ if and only if $\hat{q}<\alpha$,

\item[b.]$\alpha\leq q$  if and only if $\alpha\leq \hat{q}$,

\item[c.] $\alpha< q$  if and only if  $\alpha< \hat{q}$, and

\item[d.]$q\leq \alpha$ if and only if $\hat{q}\leq \alpha$.
\end{itemize}

\end{definition}


\begin{proposition}\label{inequ} For $q\in \Q$ and
$\alpha\in \mathbb{R}^o$,

\begin{itemize}
\item[$(1)$]$q<\alpha$ if and only if $ \exists n(q<\alpha(n))$,

\item[$(2)$] $\alpha\leq q$  if and only if $\forall
n(\alpha(n)<q)$,

\item[$(3)$]$\alpha< q$  if and only if $\exists p\in
\mathbb{Q}(p<q\wedge\alpha\leq p)$

\item[$(4)$] $q\leq \alpha$ if and only if $\forall p\in
\mathbb{Q} (p<q\rightarrow p<\alpha)$.
\end{itemize}
\end{proposition}
\begin{proof}
For each $q\in \mathbb{Q}$ consider its map $\hat{q}$ in oriented
reals. Items (1) and (2) are straight forward.

\noindent (3). If $\alpha< \hat{q}$ then by
definition~\ref{o-cut}, there exists $m$ such that for all $n$,
$\alpha(n)<q-\frac{1}{m+1}$. Let $p=q-\frac{1}{m+1}$. By (2), we
have $\alpha\leq p$. The converse is straight forward.

\noindent (4). If $\hat{q}\leq \alpha$ then for all $m$ there
exists $n_m$ such that $q-\frac{1}{m+1}<\alpha(n_m)$. For $p<q$,
there exists $m_0$ such that
$p<q-\frac{1}{m_0+1}<\alpha(n_{m_0})$. By (1), $p<\alpha$. For the
converse assume $\forall p\in \mathbb{Q} (p<q\rightarrow
p<\alpha)$. Then for all $m$, $q-\frac{1}{m+1}<\alpha$. By (1),
there exists $n_m$ such that $q-\frac{1}{m+1}<\alpha(n_m)$.
\end{proof}

For each $\alpha\in\mathbb{R}^o$, let $A_\alpha=\{q\in
\mathbb{Q}\mid \exists n (q<\alpha(n))\}$ be the \emph{cut}
specified by $\alpha$. Note that $\alpha=^o\beta$ if and only if
$A_\alpha=A_\beta$ as sets.

\begin{proposition} For $\alpha\in \mathbb{R}^o$,
\begin{enumerate}
\item $ \forall q \in\mathbb{Q}~( q\in A_\alpha\ra \exists
p\in \mathbb{Q}~ (p>q \wedge p\in A_\alpha))$
~$(\emph{openness})$,
\item $ \forall p,q\in\mathbb{Q}~(p<q\wedge
q\in A_\alpha\ra p\in A_\alpha)$ ~$(\emph{monotonicity})$,
\item $\exists p,q\in \mathbb{Q}~(p\in A_\alpha\wedge q\not\in
A_\alpha)$~$(\emph{boundedness}).$
\end{enumerate}
\end{proposition}
\begin{proof} It is straightforward.
\end{proof}

\begin{lemma}
For all $\alpha,\beta\in\mathbb{R}^o$, there exists $\gamma\in
\mathbb{R}^o$, which specifies  the cut $A_\alpha\cap A_\beta$.
\end{lemma}
\begin{proof}
The proof is easy.
\end{proof}

\begin{lemma}\label{bseq} Let $\gamma: \mathbb{N}\ra
\mathbb{Q}$ be an upper bounded sequence, i.e., $\exists M\in
\Q\forall n(\gamma(n)<M)$, then there exists $\alpha\in
\mathbb{R}^o$ which specifies $A=\{q\in \Q\mid \exists n\in
\mathbb{N}(q<\gamma(n))\}$.
\end{lemma}

\begin{proof} Let
$\alpha(n)=max\{\gamma(0),\gamma(1),...,\gamma(n)\}-\frac{1}{n+1}$.
The sequence $\alpha$ is strictly increasing and the set $A$ would
be equal to $\{q\in \Q\mid \exists n\in
\mathbb{N}(q<\alpha(n))\}$.
\end{proof}

The main difference between the collection of  oriented reals
$\mathbb{R}^o$ and other collections of  Dedekind reals such as
$\mathbb{R}^d$, extended reals $\mathbb{R}^{be}$ and  classical
reals $\mathbb{R}^e$ is that the former satisfies the
\emph{monotonicity} property ($\forall \alpha\in\mathbb{R}^o
\forall p,q\in \Q(p<q\wedge q\in A_\alpha\ra p\in A_\alpha)$),
whereas the others satisfy \emph{strong monotonicity} (for each
Dedekind cut $A$, $\forall p,q\in \Q(p<q\wedge \neg\neg q\in A\ra
p\in A)$). It is known that
$\mathbb{R}^d\subset\mathbb{R}^{be}\subset\mathbb{R}^e$
(\cite{kn:TD}, page 270). In the following proposition, by using
the Markov Principle (\cite{kn:TD}, page 204):
\begin{center}
$\MP\ \ \ \ \ \forall\beta\in\mathbb{N}^\mathbb{N}(\neg\neg\exists
k \beta(k)=0\ra\exists k \beta(k)=0)$,
\end{center}
we show that $\mathbb{R}^o\subseteq\mathbb{R}^{be}$.

\begin{proposition} {\footnotesize{$(\MP)$}} For $\alpha\in
\mathbb{R}^o$, the cut $A_\alpha$
 satisfies \emph{strong monotonicity}.
\end{proposition}

\begin{proof} We  show that for every $\alpha\in
\mathbb{R}^o$, $A_\alpha$ satisfies \emph{strong monotonicity},
i.e., $\forall p,q\in \Q(p<q\wedge \neg\neg q\in A_\alpha\ra p\in
A_\alpha)$. Since $q\in A_\alpha\leftrightarrow\exists k\in
\mathbb{N} (q<\alpha(k))$, we have $\neg\neg q\in
A_\alpha\leftrightarrow\neg\neg\exists k\in \mathbb{N}
(q<\alpha(k))$. By \MP, $\neg\neg\exists k\in \mathbb{N}
(q<\alpha(k))\leftrightarrow \exists k\in \mathbb{N}
(q<\alpha(k))$, (assume
$\beta(k)=\{^{0~~~q<\alpha(k)}_{1~~~otherwise}$). So $\neg\neg
q\in A_\alpha\leftrightarrow q\in A_\alpha$, for any arbitrary
$q\in \Q$. Then $p<q\wedge \neg\neg q\in A_\alpha\leftrightarrow
p<q\wedge q\in A_\alpha$, and since  $A_\alpha$ satisfies
\emph{monotonicity}, we derive $p \in A_\alpha$.
\end{proof}

 For choice sequences, if $\Phi$ is  a total function    from the
collection of choice sequences to natural numbers, the value of
$\Phi$ for a sequence $\alpha$ just depends on a finite segment
of $\alpha$. Now this question seems natural in constructive
mathematics:

\begin{center}
\emph{how can we \textbf{construct} a \textbf{total} mapping
$\Phi: \mathbb{R}^o\ra \mathbb{N}$?}
\end{center}

 Assume we have a strategy to construct a total mapping
$\Phi:\mathbb{R}^o\ra\mathbb{N}$. Then for any arbitrary sequence
$\alpha\in \mathbb{R}^o$, we must be able to construct a witness
for $\alpha$. Since $\Phi$ is well-defined, for any sequence
$\beta=^o\alpha$, we will have $\Phi(\beta)=\Phi(\alpha)$. Therefore, our
strategy cannot depend on any finite segment of $\alpha$. Because
of this obstacle, one may expect  that it is not possible construct
$\Phi$ unless $\Phi$ happens to be constant. We fulfill this
intention, using Brouwer's weak continuity principle
(\cite{kn:TD},~page~209):


\begin{center}
$\WCN\ \ \ \ \forall\alpha\in \T\exists y
(\Phi(\alpha)=y)\ra\forall\alpha\in\T\exists x \forall \beta\in
\T (\bar{\beta} x=\bar{\alpha} x\ra \Phi\beta=\Phi\alpha)$,
\end{center}
where $\T$ is a spread and for each sequence $\alpha$,
$\bar{\alpha}x=\langle
\alpha(0),\alpha(1),...,\alpha(x-1)\rangle$.

\begin{proposition}\label{CtoN} $(\WCN)$ Any (constructive)
total function $\Phi:\mathbb{R}^o\ra\mathbb{N}$ is constant.
\end{proposition}
\begin{proof} Let $\alpha,\beta\in \mathbb{R}^o$ be arbitrary,
$\alpha\leq\beta$, $M\in\Q$ be an upper-bound for $\beta$. The set
$\mathbb{R}^o_M=\{\gamma\in\mathbb{R}^o\mid \gamma<M\}$ is a
spread. Since $\Phi$ is a total, by $\WCN,$ there exists $t$ such
that, for all $\gamma\in\mathbb{R}^o_M$, if $\bar{\gamma}
t=\bar{\alpha}t$ then $\Phi(\gamma)=\Phi(\alpha)$. Find $\gamma\in
\mathbb{R}^o_M$ passing through $\bar{\alpha}t$ such that
$\gamma=^o\beta$. By well-definedness of $\Phi$ we conclude
$\Phi(\beta)=\Phi(\alpha)$. So, for all $\alpha,\beta\in
\mathbb{R}^o$, if $\alpha\leq\beta$ then
$\Phi(\beta)=\Phi(\alpha)$. Now, for arbitrary $\alpha,\beta\in
\mathbb{R}^o$, define $\gamma(n)=min(\alpha(n),\beta(n))$. Here,
$\gamma\leq\alpha$ and $\gamma\leq\beta$. Then
$\Phi(\gamma)=\Phi(\alpha)$ and $\Phi(\gamma)=\Phi(\beta)$.
Consequently $\Phi(\beta)=\Phi(\alpha)$.
\end{proof}

 We showed that any total function from $\mathbb{R}^o $ to natural
numbers is constant in the presence of the Weak Continuity
Principle, that is., if one has a constructive method that for each
oriented cut is able to introduce a witness, a natural number,
then the witness is unique. To avoid this, as a replacement for
natural numbers, we consider another category of objects called
{\em almost natural numbers}, and construct witnesses for oriented
reals from this category.

\begin{definition} We let $\mathbb{N}^\ast$ be the set of all functions $\xi$ from $\mathbb{N}$
to $\mathbb{N}$ such that, for some $k$, for all $n$, $\xi(n)
\leq \xi(n + 1) \leq k$. For all $\xi, \nu$ in $\mathbb{N}^\ast$
we define:

\begin{itemize}\item[]
$\xi \leq \nu$ if and only if $\forall m \exists n (\xi(m) \leq
\nu(n))$, and

\item[] $\xi =^\ast \nu$ if and only if $(\xi \leq \nu)\wedge (\nu
\leq \xi)$.
\end{itemize}  We call $\mathbb{N}^\ast$,
regarding equality $=^\ast$, the set of \emph{almost natural
numbers}.
\end{definition}

It is clear that $\mathbb{N}^\ast$ is classically isomorphic to $
\mathbb{N}$ as a set. In fact, classically, elements in
$\mathbb{N}^\ast$ are increasing sequences  that converge.

\begin{lemma}\label{converge}
For every $\xi\in \mathbb{N}^\ast$,
\begin{enumerate}
\item $\neg\neg\exists n\forall
m>n[\xi(m)=\xi(n)]$,
\item $\neg\forall n \exists m>n~ (\xi(m)\neq\xi(m+1))$.
\end{enumerate}
\end{lemma}

\begin{proof}
\begin{enumerate}
\item To show $\forall\xi\in \mathbb{N}^\ast\neg\neg\exists
n\forall m>n[\xi(m)=\xi(n)]$, it is enough to show
$\neg\exists\xi\in \mathbb{N}^\ast\neg\exists n\forall
m>n[\xi(m)=\xi(n)]$, by the intuitionistic valid statement
$\neg\exists x A(x)\leftrightarrow\forall x\neg A(x)$. So assume
for some $\xi_0\in \mathbb{N}^\ast$, $\neg\exists n\forall
m>n[\xi_0(m)=\xi_0(n)]$. Then, $\forall n\neg\forall
m>n[\xi_0(m)=\xi_0(n)]$.

\noindent  Since $\xi_0\in \mathbb{N}^\ast$, there exists
$k_0\in\mathbb{N}$ such that $\xi_0(n)\leq\xi_0(n+1)\leq k_0$, for
all $n\in\mathbb{N}$.

\noindent We prove that for every $k\in \mathbb{N}$,
 $ (\forall
n(\xi_0(n)\leq k))\rightarrow(\forall n(\xi_0(n)< k))$~(1).

\noindent Assume there exists $t\in \mathbb{N}$ such that
$\xi_0(t)=k$. Since $\xi_0$ is nondecreasing and $\forall
n(\xi_0(n)\leq k))$, we have $\forall n>t[\xi_0(n)= k]$. It
contradicts with $\forall n\neg\forall m>n[\xi_0(m)=\xi_0(n)]$.
Hence $\forall n(\xi_0(n)< k)$.

\noindent Now let  $k=k_0$. By (1), we derive $\forall
n(\xi_0(n)\leq k_0-1)$. Repeating using (1), we have got $\forall
n(\xi_0(n)= 0)$. It contradicts with $\forall n\neg\forall
m>n[\xi_0(m)=\xi_0(n)]$.

\item Let $\xi\in \mathbb{N}^\ast$, there exists $k\in \mathbb{N}$
such that for all $n$, $\xi(n)<k$. Assume
\begin{center} $\forall
n \exists m>n~ (\xi(m)\neq\xi(m+1))$. (2)
\end{center}
Define  $f:\mathbb{N}\rightarrow \mathbb{N}$ as
\begin{center}
$f(n)=\min\{m\mid (m>n \wedge \xi(m)\neq\xi(m+1))\}+1$.
\end{center}
Due to the assumption~(2), the function $f$ is a constructive well
defined function. It is easy to see that $n<f(n)$ and
$\xi(n)<\xi(f(n))$. Therefore,\begin{center}
$\xi(1)<\xi(f(1))<\xi(f(f(1)))<\xi(f(f(f(1))))<...$.
\end{center}It contradicts with the fact that for all $n$,
$\xi(n)<k$.
\end{enumerate}
\end{proof}

  It is worth mentioning that if $\gamma$ is an almost natural
number then the set $I=\{n\in \mathbb{N}\mid \exists k~ (n\leq
\gamma(k))\}$ is not necessary finite, but it is
\emph{quasi-finite} in the sense of \cite{kn:vel}, i.e., it is a
subset of a finite set. Moreover, its complement is \emph{almost
full}~\cite{kn:vel}, meaning that  for every strictly increasing
sequence $\alpha$, one may find a natural number $n$ such that
$\alpha(n)$ belongs to the complement of $I$.

 Until now, the two notions of oriented reals and almost natural
numbers have been defined. The following proposition (accompanied by its
proof) shows that how these two notions are related to our
intuition of  orientation. In contrast with proposition
\ref{CtoN}, we have

\begin{proposition}
There exists a non-constant total function $\Phi: \mathbb{R}^o\ra
\mathbb{N}^\ast$.
\end{proposition}

\begin{proof} Choose $d_1< d_2< ...< d_j$ from $\Q$ for
some fixed $j\in \mathbb{N}$. For any $\beta\in \mathbb{R}^o$, we
define
\begin{itemize}
\item[] $\Phi(\beta)(0)=0$, and \item[]
$\Phi(\beta)(n)=max(\{i\mid i\leq
j\wedge(d_i\leq\beta(n))\}\cup\{0\})$ for all $n\geq 1$.
\end{itemize}
One may easily  check that the followings hold true for $\Phi$:
\begin{enumerate}
\item $\Phi(\beta)\in  \mathbb{N}^\ast$. \item $\Phi$ is
well-defined, i.e., if $\alpha=^o\beta$   then
$\Phi(\alpha)=^\ast\Phi(\beta)$.

\item $\Phi$ is non-constant. Assume two different oriented cuts
$\alpha, \beta$, such that $d_1\not\in A_\alpha$ and $d_1\in
A_\beta, d_2\not\in A_\beta$. Then $\forall n\in \mathbb{N}$
$\Phi(\alpha)(n)=0$, whereas, $\exists k\forall n>k$
$\Phi(\beta)(n)=1$.
\end{enumerate}
The function $\Phi$  can be computed by the following algorithm as
well: {\footnotesize{
\begin{itemize}
\item[1.] Put $\Phi(\beta)(0):=0$; \item[2.] put~$ t=1$; \item[3.]
For i=1 to j~ do:\begin{itemize}\item[] $\{$ \item[]while $(
\beta(t)< d_i)$ do:\begin{itemize} \item[] $\{$ \item[]define $
\Phi(\beta)(t):=i-1$; \item[]~put
$t=t+1$;\item[]$\}$\end{itemize}\item[]$\}$\end{itemize} \item[4.]
For k=t+1 to $\infty$ do: \begin{itemize}\item[] $\{$ \item[] Put
$ \Phi(\beta)(k):=j$;\item[]$\}$\end{itemize}
\end{itemize}}}
The above algorithm checks the membership status of $d_i$'s  in
$A_\beta$ respectively, in an ordered manner. As soon as, the
algorithm detects that $d_i$ is in $A_\beta$, the future value of
$\Phi(\beta)$ changes. We defined $\Phi$  using the above
algorithm, in order to evoke our sense of the orientation
explained in the  introduction; the oriented real number $\beta$
is assumed to be a moment in future. Rational numbers
$d_1<d_2<...<d_j$ are assumed as future moments which we examine
the occurrence of $\beta$ at them, as time passing. Using the
information attained by  the examination, the value of
$\Phi(\beta)\in \mathbb{N}^\ast$ is determined.
\end{proof}

  Therefore there exists a total function from $\mathbb{R}^o$ to
$\mathbb{N}^\ast$ which is not constant. \emph{How a constructive
{\em total} function form $\mathbb{R}^o$ to $\mathbb{N}^\ast$ can
be? or in other word, how can we construct a {\em total} mapping
$\Phi: \mathbb{R}^o \ra \mathbb{N}^\ast$?} We answer this
question  soon in section~\ref{tocp}.

In the following, we illustrate  that, although  the oriented reals
can not be approximated by rationals, they can be approximated by
{\em almost rational numbers}, defined below.

\begin{definition}
Let  $\mathbf{Q}$  be the set of all rational sequences $\zeta:
\mathbb{N}\ra \mathbb{Q}$  in which

\begin{itemize}
\item[1.] $\zeta$ is increasing, and

\item[2.] the image of $\zeta$ is a subset of a finite set.
\end{itemize}

For all $\zeta,\zeta'\in \mathbf{Q}$, we define:
\begin{itemize}
\item[]$\zeta\leq\zeta'$ if and only if $\forall m\exists
n(\zeta(m)\leq\zeta'(n))$, and

\item[]$\zeta =\zeta'$ if and only if
$(\zeta\leq\zeta')\wedge(\zeta'\leq\zeta)$.
\end{itemize}
  We call $\mathbf{Q}$ the set of \emph{almost
rational numbers}.
\end{definition}
Note that $\mathbf{Q}$ is classically isomorphic to $\mathbb{Q}$
as a set. The set of almost rational numbers $\mathbf{Q}$ is
embedded into $\mathbb{R}^o$, by $\zeta\mapsto\hat{\zeta}$ where
$\hat{\zeta}(n)=\zeta(n)-\frac{1}{n+1}$, for $n\in \mathbb{N}$, is
an oriented real specifying the cut $\{q\in \mathbb{Q}\mid \exists
k (q<\zeta(k))\}$.

\begin{definition} For every $\alpha \in \mathbb{R}^o$,
and $r\in \mathbb{Q}$, define $(\alpha +r)\in \mathbb{R}^o$, by
$(\alpha +r)(n):=\alpha(n)+r$, for every $n$.
\end{definition}

\begin{definition} We say that an oriented real $\alpha$ can be approximated by rationals,
if for each $n$, there exists $q\in \mathbb{Q}$ such that
$(q\leq\alpha\wedge \alpha\leq q+ 2^{-n})$.
\end{definition}

  As we explained in the introduction, since moments of  time
sink back into the past, it is not possible to  approximate them using
rationals. The following proposition shows that oriented reals
cannot be approximated by rational numbers either.

\begin{proposition} {\footnotesize{$(\neg\exists$-PEM$)$}} It is false that
every oriented real can be approximated by rationals.

\end{proposition}

\begin{proof}  For an arbitrary  bounded increasing sequence $\alpha$ of
rationals, let $\beta_\alpha$ be an oriented real defined by
$\beta_\alpha(n)=\alpha(n)-\frac{1}{n+1}$, for $n\in \mathbb{N}$,
which specifies the cut $A=\{q\in \Q\mid \exists
k(q<\alpha(k))\}$. If $\beta_\alpha$ can be approximated by
rationals, then $\beta_\alpha$ is a Cauchy sequence. Consequently,
the sequence $\alpha$ is also Cauchy and so it converges. Since
$\alpha$ was arbitrary, it would imply that every bounded monotone
sequence has a limit in Brouwerian real line. That contradicts
with $\neg\exists$-PEM~(\cite{kn:TD}, page~268).
\end{proof}

\begin{definition} We say that an oriented real $\alpha$ can be approximated
by almost rationals, if $\forall n \exists\zeta\in
\mathbf{Q}(\hat{\zeta}  \leq \alpha\wedge \alpha\leq (\hat{\zeta}
+2^{-n}))$.
\end{definition}

\begin{proposition}~\label{app}
Every oriented real can be approximated by almost rationals
numbers.
\end{proposition}

\begin{proof}
Let  $\beta\in\mathbb{R}^o$, and $M\in \mathbb{Q}$ be an  upper
bound for $\beta$. We define $\zeta$ by induction. Let
$\zeta(0)=\beta(0)$, and assume we have defined $\zeta(k)$. We
define
\begin{center}
$\zeta(k+1)=\zeta(k)+2^{-n}t$, where  $t\in\mathbb{N} $ and
$\zeta(k)+2^{-n}t\leq\beta(k+1) <\zeta(k)+2^{-n}(t+1)$.
\end{center}
Note that $\zeta\in\mathbf{Q}$, since
$image(\zeta)\subseteq\{\beta(0)+2^{-n}t\mid
\beta(0)+2^{-n}t<M,t\in \mathbb{N}\}$ and the latter  is not
infinite. We claim that $\hat{\zeta} \leq \beta \wedge
\beta\leq\hat{\zeta} +2^{-n}$. For $\hat{\zeta}\leq \beta$, we
have for each $k\in \mathbb{N}$ there exists $m$ such that
$\zeta(k)< \beta(m)$. For the second clause, i.e.,
$\beta\leq\hat{\zeta} +2^{-n}$,  we have
$\beta(k)<\zeta(k)+2^{-n}$.\end{proof}

A  subset $B$ of $\mathbb{N}$ is called (intuitionistically)
enumerable if it is the image of a function on natural numbers
(see~\cite{kn:ARR}). Assuming  the Kripke schema (\cite{kn:TD},
page 236):
\begin{center}
$\KS\ \ \ \ \forall X\exists\alpha\forall n[n\in
X\leftrightarrow\exists m (\alpha(n,m)=0)]$.
\end{center}
every inhabited~\footnote{A set $B$ is inhabited if $\exists
x(x\in B)$. We indicate that by $B\#\emptyset$.} subset of natural
numbers is (intuitionistically) enumerable~\cite{kn:ARR}.

\begin{lemma}\label{l5} If $B\subseteq \mathbb{Q} $
is (intuitionistically) enumerable and bounded from above, then
there exists an oriented real number $\alpha$ such that
$A_\alpha=C(B) =\{q\in\Q\mid \exists p (p\in B\wedge q<p)\}$.
\end{lemma}

\begin{proof} As $B\subseteq \mathbb{Q} $
is (intuitionistically) enumerable, there exists a sequence
$\gamma:\mathbb{N} \ra \mathbb{Q}$ such that $\forall q\in \Q[q\in
B\leftrightarrow \exists m(\gamma(m)=q)]$. Since $B$ is bounded
from above, the sequence $\gamma$ has an upper bound. By
lemma~\ref{bseq}, there exists an oriented number $\alpha$ such
that $A_\alpha=C(B)$.
\end{proof}


Assume $B\subseteq \mathbb{Q} $ is (intuitionistically) enumerable
and upper bounded. We define the supremum of $B$, $sup(B)$, to be
the oriented real number $\alpha$  that $A_\alpha=C(B)$. Assuming
the Kripke schema, $\KS$, every upper bounded subset of
$\mathbb{Q}$ has  supremum.

\begin{proposition}$(\KS).$\label{psupremum}
Assume $B\subseteq \mathbb{Q}$ is upper bounded. Then for
$\alpha=sup(B)$, the following hold.

\begin{itemize}
\item[1.] $\forall q\in \mathbb{Q}(q\in B\rightarrow q\leq
\alpha)$,

\item[2.] $\forall p\in \mathbb{Q}(p<\alpha\rightarrow \exists
q\in \mathbb{Q}(q\in B\wedge p<q))$.

\end{itemize}
\end{proposition}
\begin{proof} (1.) Assume $p\in B$. For any
$q<p$, we have $q\in C(B)=A_\alpha$. Then there exists $n$ such
that $q<\alpha(n)$. By items $1$ and $4$ of
proposition~\ref{inequ}, we have $p\leq \alpha$.

\noindent (2.) Assume $p<\alpha$. Then $p\in A_\alpha$. As
$A_\alpha= C(B)$, there exists $q\in B$ such that $p<q$.
\end{proof}


\begin{definition}\label{infm} Let $D\subseteq \mathbb{Q}$ is lower bounded, and
$B=\{p\in\mathbb{Q}\mid \forall q (q\in D\ra q\geq p)\}$. We
define the infimum of $D$, $inf(D)$, to be the supremum of $B$.
\end{definition}

\begin{proposition}$(\KS).$ Assume $D\subseteq \mathbb{Q}$ is lower
bounded. Then for $\alpha=inf(D)$, the following hold.

\begin{itemize}
\item[1.] $\forall q\in \mathbb{Q}(q\in D\rightarrow \alpha\leq
q)$,

\item[2.] $\forall p\in \mathbb{Q}[(\forall q\in \mathbb{Q}(q\in
D\rightarrow p\leq q))\rightarrow p\leq \alpha ]$.

\end{itemize}
\end{proposition}

\begin{proof} (1.) $\alpha=inf(D)$ is an oriented real such that $A_\alpha=C(B)
=\{q\in\Q\mid \exists p (p\in B\wedge q<p)\}$, where
$B=\{p\in\mathbb{Q}\mid \forall q (q\in D\ra q\geq p)\}$. Let
$q\in D$. Since $A_\alpha=\{q\in\Q\mid \exists n(q<\alpha(n))\}$
and $\alpha$ is strictly increasing, we have $\alpha(n)\in
A_\alpha$, for every $n$. By the equality $A_\alpha=C(B)$, it is
derived that for each $n$, there exists $p\in B$ such that
$\alpha(n)<p$. Then by definition of $B$, $\alpha(n)<q$.

(2.) Assume $p\in \mathbb{Q}$ is such that for all $q\in D$,
$p\leq q$. Then $p\in B$, and thus for all $p_0<p$, $p_0\in
C(B)=A_\alpha$. According to definition of $A_\alpha$, there
exists $n$ such that $p_0<\alpha(n)$. Then, $(\forall
p_0<p)\exists n (p_0<\alpha(n))$. By items $1$ and $4$ of
proposition~\ref{inequ}, we have $p\leq \alpha$.\end{proof}


\begin{theorem}$(\KS)$. \textbf{\emph{The monotone convergence theorem}}.
For every upper bounded nondecreasing sequence $(\alpha_n)_{n\in
\mathbb{N}}$ of oriented reals, there exists an oriented real
$\alpha$ such that

\begin{itemize}
\item[$1.$] $\forall n (\alpha_n\leq \alpha)$,

\item[$2.$] $(\forall p\in \mathbb{Q})[(p<\alpha)\rightarrow
(\exists m\forall n(n\geq m\rightarrow p<\alpha_n))]$.
\end{itemize}
\end{theorem}

\begin{proof} Let $B=\{p\in \mathbb{Q}\mid \exists n
(p<\alpha_n)\}$. The set $B$ is upper bounded, so let
$\alpha=sup(B)$. For all $n,m\in \mathbb{N}$, we have
$\alpha_n(m)\in B$, since $\alpha_n$ is a strictly decreasing
sequence of rationals. By proposition~\ref{psupremum}, we have for
each $n$, for all $m$, $\alpha_n(m)<\alpha_n(m+1)\leq\alpha$. So,
by items (1) and (4) of proposition~\ref{inequ}, for each $n$,
$\alpha_n\leq\alpha$.

\noindent If $p<\alpha$, then by proposition~\ref{psupremum},
there exists $q\in B$ such that $p<q$. By definition of $B$, there
exists $m\in \mathbb{N}$, $p<q<\alpha_m$. As $(\alpha_n)_{n\in
\mathbb{N}}$ is a non-decreasing sequence, we have for all $n\geq
m$, $p<\alpha_n$.
\end{proof}

Among the Dedekind lines $\mathbb{R}^e$, $\mathbb{R}^{be}$, and
$\mathbb{R}^d$ introduced in~\cite{kn:TD}, the line
$\mathbb{R}^{be}$ is much similar to $\mathbb{R}^o$. The only
difference between the cuts of $\mathbb{R}^{be}$ and the cuts of
$\mathbb{R}^o$ is that the first one satisfies \emph{strong
monotonicity},
whereas the second one just fulfils monotonicity. On the other
hand, as we have already noted, the reals in $\mathbb{R}^{be}$
cannot be approximated by rationals, as   desired   as a
requirement for the temporal line. Hence, the line
$\mathbb{R}^{be}$ could be assumed as an appropriate mathematical
model for the temporal line if there was a \emph{continuity
principle} for it, like Brouwer's continuity principle for
\emph{Cauchy} reals $\mathbb{R}$. We note that the line
$\mathbb{R}$ is the \emph{Cauchy} completion of $\mathbb{Q}$ and
the line $\mathbb{R}^{be}$ is the \emph{order completion} of
$\mathbb{Q}$.

Let us define a mapping $\Psi: \mathbb{R}^o\rightarrow
\mathbb{R}^{be}$ as follows:
\begin{center}
$\Psi(\alpha):=\{r\in \mathbb{Q}\mid (\exists s\in
\mathbb{Q})[r<s\wedge \neg\neg\exists n (s<\alpha(n))] \}$.
\end{center}
One can easily check that  $\Psi$ is well-defined and  for all
$\alpha\in \mathbb{R}^o$, $\Psi(\alpha)\in \mathbb{R}^{be}$.

\begin{proposition}$(\KS)$. The mapping $\Psi: \mathbb{R}^o\rightarrow
\mathbb{R}^{be}$ is surjective.
\end{proposition}

\begin{proof}
Assuming the Kripke schema,
it can be shown that every inhabited subset of natural numbers is
(intuitionistically) enumerable~\cite{kn:ARR}.

Let $S\in \mathbb{R}^{be}$, and $\gamma:\mathbb{N} \ra S$
enumerates $S$, i.e., $\forall q\in \mathbb{Q}[ q\in
S\leftrightarrow \exists m(\gamma(m)=q)]$. Define
$$\alpha(n)=max\{\gamma(0),\gamma(1),...,\gamma(n)\}-\frac{1}{n+1}.$$
The sequence $\alpha$ is in $\mathbb{R}^o$. We claim
$\Psi(\alpha)=S$.

Assume $r\in \Psi(\alpha)$. Then there exists $s\in \mathbb{Q}$
such that $r<s$ and  $\neg\neg \exists n(s<\alpha(n))$. For the
sake of argument, assume $ \exists n(s<\alpha(n))$. Note that for
all $m$, $\alpha(m)\in S$. By the \emph{monotonicity} of $S$, we
have $s\in S$. Thus $\neg\neg s\in S$. By the \emph{strong
monotonicity}, we have $r\in S$.

For the converse, assume $r\in S$.  Let $k$ be such that
$\gamma(k)=r$. By openness, there exists $r'\in S$ such that
$r<r'$, and let  $\gamma(m)=r'$ for some $m$. Choose
 $n>m$ such that $r'-r<\frac{1}{n+1}$. Then we have
$r<\alpha(n)$ and thus $r\in \Psi(\alpha)$.
\end{proof}

\begin{proposition} For all $\alpha,\beta\in \mathbb{R}^o$
\begin{center}
$\alpha<\beta$ implies $\Psi(\alpha)<\Psi(\beta)$.
\end{center}
\end{proposition}
\begin{proof}
The relation ``$<$" for cuts $S,T$ is defined as follows:
$S<T:=\exists r>0(S+r\subset T)$ (see definition~5.4
of~\cite{kn:TD}). Suppose $\alpha<\beta$. Then for some $n$, we
have $\forall m(\alpha(m)<\beta(n))$. Let
$r=\beta(n+1)-\beta(n)$. It is easy to check that
$\Psi(\alpha)+r\subset\Psi(\beta)$.
\end{proof}

\subsection{Arithmetic in $\mathbb{R}^{o}$}

As explained in the introduction, we propose $\mathbb{R}^{o}$
as a model for the temporal line. Each oriented real illustrates
a moment. Then what does it mean to add or multiply two moments?
What is a proper arithmetic of the temporal line?

The real line $\mathbb{R}^{be}$ is a \emph{field} due to the
\emph{strong monotonicity} property of its elements. A field
$\langle F,+,\cdot\rangle$ is an algebraic structure with two
\emph{functions} $+$ and $\cdot$ from $F\times F$ to $F$, where
$F$ equipped with the functions $+$ or $\cdot$, is a group, and
$\cdot$ is distributed over $+$.  In our model, $\mathbb{R}^{o}$,
the operations $+$ and $\cdot$ are \emph{relations} instead of
being functions, i.e, it is an algebraic structure known as a\emph{hyperstructure}~\footnote{For similar definitions and
applications of hyper-algebraic structures as arithmetic,
see~\cite{kn:krasn,kn:conn,kn:conn2}.}.

\begin{definition} A \emph{h-field} is a tuple $\langle
F,+,\ast,0,1\rangle$ where $+\subseteq F\times F\times F$ and
$\ast\subseteq F\times F\times F$ satisfy the following axioms:
\begin{itemize}
\item[-] \emph{Inhabitance}.
\[
\begin{array}{l}\begin{array}{ccccc}
 (\forall x,y)(\exists z) +(x,y,z) & &(\forall x,y)(\exists z )
 \ast(x,y,z).
\end{array}\quad\quad\quad\quad\quad\quad\quad\quad\quad\quad\quad\quad\quad
\end{array} \]
\item[-] \emph{Identity}.
\[
\begin{array}{l}\begin{array}{ccccc}
 (\forall x) +(x,0,x) & &(\forall x )
 \ast(x,1,x).
\end{array}\quad\quad\quad\quad\quad\quad\quad\quad\quad\quad\quad\quad\quad
\end{array} \]
\item[-] \emph{Inverse}.
\[
\begin{array}{l}\begin{array}{ccccc}
 (\forall x\exists y) +(x,y,0) & &(\forall x\exists y )
 \ast(x,y,1).
\end{array}\quad\quad\quad\quad\quad\quad\quad\quad\quad\quad\quad\quad\quad
\end{array} \]
\item[-]\emph{Commutativity}.
\[
\begin{array}{l}\begin{array}{ccccc}
 (\forall x,y,z)( +(x,y,z)\leftrightarrow +(y,x,z)) & &
 (\forall x,y,z)( \ast(x,y,z)\leftrightarrow \ast(y,x,z)).
\end{array}\quad\quad\quad\quad\quad\quad\quad\quad\quad\quad\quad\quad\quad
\end{array} \]

\item[-] \emph{Associativity}.
\[
\begin{array}{l}\begin{array}{ccccc}
 (\forall x,y,z,w,v,u)(( +(x,y,w)\wedge +(w,z,v)\wedge +(y,z,u))\rightarrow +(x,u,v))
 \\
 (\forall x,y,z,w,v,u)(( +(x,y,w)\wedge +(x,u,v)\wedge +(y,z,u))\rightarrow +(w,z,v))
\end{array}\quad\quad\quad\quad\quad\quad\quad\quad\quad\quad\quad\quad\quad
\end{array} \]

\[
\begin{array}{l}\begin{array}{ccccc}
 (\forall x,y,z,w,v,u)(( \ast(x,y,w)\wedge \ast(w,z,v)\wedge \ast(y,z,u))\rightarrow \ast(x,u,v))
 \\
 (\forall x,y,z,w,v,u)(( \ast(x,y,w)\wedge \ast(x,u,v)\wedge \ast(y,z,u))\rightarrow
 \ast(w,z,v)).
\end{array}\quad\quad\quad\quad\quad\quad\quad\quad\quad\quad\quad\quad\quad
\end{array} \]

\item[-]\emph{Distributivity}.
\[
\begin{array}{l}\begin{array}{ccccc}
 (\forall x,y,z,w,v,u,r)(( \ast(x,v,w)\wedge +(y,z,v)\wedge \ast(x,y,u)\wedge
  \ast(x,z,r))\rightarrow +(u,r,w))
 \\
 (\forall x,y,z,w,v,u,r)((+(u,r,w) \wedge +(y,z,v)\wedge \ast(x,y,u)\wedge
  \ast(x,z,r))\rightarrow \ast(x,v,w)).
\end{array}\quad\quad\quad\quad\quad\quad\quad\quad\quad\quad\quad\quad\quad
\end{array} \]
\end{itemize}
\end{definition}
We define an \emph{addition relation} $+$, and a
\emph{multiplication relation} $\ast$ on $\mathbb{R}^o$ as
follows:

For $\alpha,\beta,\gamma\in \mathbb{R}^o$, we let

\begin{itemize}
\item[1.] $+(\alpha,\beta,\gamma)$ if and only if
$\Psi(\alpha)+\Psi(\beta)=\Psi(\gamma)$, and

\item[2.]$\ast(\alpha,\beta,\gamma)$ if and only if
$\Psi(\alpha)\cdot\Psi(\beta)=\Psi(\gamma)$.
\end{itemize}

\begin{proposition}$(\KS)$. $\langle\mathbb{R}^o,+,\ast,\hat{0},\hat{1}\rangle$ is a h-field.

\end{proposition}
\begin{proof}
The proof is straightforward by using the fact that
$\mathbb{R}^{be}$ is a field.
\end{proof}

\subsection{The Oriented Continuity Principle}\label{tocp}

  In this part, we propose  the oriented continuity principle which
expresses formally our sense of the notion of   orientation. To
do this, we  study total functions $\Phi$ from $(0,1]^o$ to
$\mathbb{N}^\ast$, where $(0,1]^o=\{\alpha\in \mathbb{R}^o\mid
\hat{0}<\alpha\leq \hat{1}\}$. Similarly, we let
$[0,1]^o=\{\alpha\in\mathbb{R}^o\mid \hat{0}\leq\alpha\leq
\hat{1}\}$. We assume such function $\Phi$ is well-defined, i.e.,
$\alpha=^o\beta$ implies $\Phi(\alpha)=^\ast\Phi(\beta)$.

\begin{lemma}\label{l4} $(\WCN)$. Let $\Phi\in(0,1]^o\ra
\mathbb{N}^\ast$ be total. Then $\forall \alpha,\beta\in(0,1]^o
(\alpha\leq\beta\ra \Phi(\alpha)\leq \Phi(\beta))$.
\end{lemma}
\begin{proof} For $\theta\in \mathbb{N}^\ast$ define $N_\theta=
\{m\in \mathbb{N}\mid \exists n (m\leq \theta(n))\}$. We need to
show that   $\alpha\leq \beta$ implies $N_{\Phi(\alpha)}\subseteq
N_{\Phi(\beta)}$. Assume $\Phi(\alpha)(n)=k$ for some
$n,k\in\mathbb{ N}$. The set $(0,1]^o$ is a spread and $\Phi$ is
total, so by \WCN we can find a $t$ such that for each $\delta\in
(0,1]^o$, $\bar{\delta}t=\bar{\alpha}t$ implies
$\Phi(\delta)(n)=\Phi(\alpha)(n)$. We have,  for each $\delta\in
(0,1]^o$, if $\bar{\delta}t=\bar{\alpha}t$ then $k\in
N_{\Phi(\delta)}$. Since $\alpha\leq \beta$ there exists
$\lambda\in (0,1]^o$ such that $\bar{\lambda}t=\bar{\alpha}t$ and
$\lambda=^o\beta$. Note that $\Phi$ is well-defined, therefore
$k\in N_{\Phi(\beta)}$. Since we assumed $k\in \mathbb{N}$ to be
arbitrary, we have $N_{\Phi(\alpha)}\subseteq N_{\Phi(\beta)}$.
\end{proof}

The following theorem is by W.~Veldman from our correspondence
with him.

\begin{theorem}\label{wim}$(\WCN)$.
Let $\Phi$ be a total well-defined function from $(0,1]^o$ to
$\mathbb{N}^\ast$. Then
\begin{center}
$\forall\alpha\in (0,1]^o\neg\neg\exists q\in
\mathbb{Q}[q<\alpha\wedge\forall\beta\in
(0,1]^o[(q<\beta\leq\alpha\ra \Phi(\alpha)=^\ast\Phi(\beta)]]$.
\end{center}
\end{theorem}
\begin{proof} Let $\alpha\in (0,1]^o$. Since
$\Phi(\alpha)\in \mathbb{N}^\ast$, by lemma \ref{converge},
$\neg\neg\exists n\forall m>n[\Phi(\alpha)(m)=\Phi(\alpha)(n)]$.
For the sake of the argument, assume  $\exists n\forall
m>n[\Phi(\alpha)(m)=\Phi(\alpha)(n)]$, and let $n_0$ be such that
$\forall m>n_0[\Phi(\alpha)(m)=\Phi(\alpha)(n_0)]$. By $\WCN$,
there is a $t$ such that for all $\beta$ in the spread $(0,1]^o$,
if $\bar{\beta}t=\bar{\alpha}t$ then
$\Phi(\beta)(n_0)=\Phi(\alpha)(n_0)$. It follows that, for all
$\beta\in(0,1]^o$ passing through
$\bar{\alpha}t=\langle\alpha0,\alpha1,...,\alpha(t-1)\rangle$,
$\Phi(\alpha)\leq\Phi(\beta)$.

\noindent Define $q:=\alpha(t-1)$. Assume that $\beta\in(0,1]^o$
and $q<\beta\leq\alpha$. Find $\lambda$ passing through
$\bar{\alpha}t$ such that $\lambda=^o\beta$,  and conclude
$\Phi(\alpha)\leq \Phi(\lambda)=^\ast\Phi(\beta)$. 
On the other hand, by lemma~\ref{l4}, since $\beta\leq\alpha$ we
have $\Phi(\beta)\leq\Phi(\alpha)$. Hence, for every $\beta$
satisfying $q<\beta\leq\alpha$, $\Phi(\beta)=^\ast\Phi(\alpha)$.
This conclusion is obtained from the assumption: $\exists n\forall
m>n[\Phi(\alpha)(m)=\Phi(\alpha)(n)]$. As we know $\neg\neg\exists
n\forall m>n[\Phi(\alpha)(m)=\Phi(\alpha)(n)]$, we may conclude
$\forall\alpha\in (0,1]^o\neg\neg\exists q\in
\mathbb{Q}[q<\alpha\wedge\forall\beta\in
(0,1]^o[(q<\beta\leq\alpha\ra \Phi(\alpha)=^\ast\Phi(\beta)]]$.
\end{proof}


Note that for any $\Phi: (0,1]^o\ra \mathbb{N}^\ast$, there
exists $k\in \mathbb{N}$ such that for all $n$,
$\Phi(\hat{1})(n)\leq k$, by definition of almost natural numbers.
\begin{proposition}\label{Tproperty}$(\WCN)$.
Let $\Phi: (0,1]^o\ra \mathbb{N}^\ast$ be total, $k\in \mathbb{N}$
be such that for all $n$, $\Phi(\hat{1})(n)\leq k$, and
$T_i=\{q\in\mathbb{Q}\mid \Phi(\hat{q})=^\ast \underline{i}\}$,
for $0\leq i\leq k$, where $\underline{i}=\langle
i,i,i,...\rangle\in \mathbb{N}^\ast$. Then

\begin{itemize}
\item[(a).] if $i<j$, $q\in T_i$ and $p\in T_j$, then $q<p$,
\item[(b).] for each $\alpha\in(0,1]^o$, if $A_\alpha\cap
T_i\#\emptyset$, then $\Phi(\alpha)\geq \underline{i}$,
\item[(c).] for each $\alpha\in(0,1]^o$, if $\Phi(\alpha)=^\ast
\underline{i}$, then $\neg (A_\alpha\cap T_i=\emptyset)$,
\item[(d).] for each $\alpha\in(0,1]^o$, $\neg\neg\exists
i[(0\leq i\leq k)\wedge (\Phi(\alpha)=^\ast \underline{i})]$.
\end{itemize}
\end{proposition}
\begin{proof} (a) and (b) are derived by lemma~\ref{l4}.
(c) follows from theorem \ref{wim}, and (d) is a consequence of
the definition of almost natural numbers, and the fact
$\forall\alpha\in(0,1]^\ast(\Phi(\alpha)\leq \Phi(\hat{1})\leq
\underline{k})$.
\end{proof}

\begin{theorem} \label{theorem-ocp}$(\WCN+\KS)$. Assume
$\Phi:(0,1]^o\ra\mathbb{N}^\ast$. Let $\delta_i=inf(T_i)$, for
$T_i$s $0\leq i\leq k$, defined above, and $E=\{\delta_i\mid
0\leq i\leq k\}$. For each $\alpha \in(0,1]^o$, define
$L_\alpha=\{\gamma\in(0,1]^o\mid \gamma<\alpha\}$. Then
\begin{center} $\forall \alpha,\beta \in(0,1]^o [(L_\alpha\cap
E=L_\beta\cap E)\ra\neg\neg(\Phi(\alpha)=^\ast \Phi(\beta))]$.
\end{center}
\end{theorem}
\begin{proof} First, observe that for $\alpha\in (0,1]^o$
and $0\leq i\leq k$,
\begin{itemize}
\item[(a)] $\delta_i\in L_\alpha\ra \neg (A_\alpha\cap
T_i=\emptyset)$, and
\item[(b)] $A_\alpha\cap T_i\#\emptyset\ra\delta_i\in
L_\alpha$
\end{itemize}
To prove the theorem, let $\alpha, \beta\in (0,1]^o$ such that
$L_\alpha\cap E=L_\beta\cap E$ and $\neg(\Phi(\alpha)=^\ast
\Phi(\beta))$~(1). we want to derive a contradiction. By
proposition \ref{Tproperty}(d), we have $\neg\neg\exists i[(0\leq
i\leq k)\wedge (\Phi(\alpha)=^\ast \underline{i})]$~(2), and
$\neg\neg\exists i[(0\leq i\leq k)\wedge (\Phi(\beta)=^\ast
\underline{i})]$~(3). Applying the intuitionistic  valid statement
$\neg\neg(\varphi\wedge\psi)\leftrightarrow\neg\neg\varphi\wedge\neg\neg\psi$
to (1), (2) and (3), gives $\neg\neg (\exists i,j[(0\leq i,j\leq
k)\wedge (\Phi(\alpha)=^\ast \underline{i})\wedge
(\Phi(\beta)=^\ast \underline{j})\wedge \neg(\Phi(\alpha)=^\ast
\Phi(\beta)) ])$, which implies
\begin{center}
$\neg\neg( \exists i,j[(0\leq i,j\leq k)\wedge (\Phi(\alpha)=^\ast
\underline{i})\wedge (\Phi(\beta)=^\ast \underline{j})\wedge
\neg(\underline{i}=^\ast \underline{j}) ])$.
\end{center}
Now, for the sake of argument, assume
\begin{center}$\psi\equiv\exists
i,j[(0\leq i,j\leq k)\wedge (\Phi(\alpha)=^\ast
\underline{i})\wedge (\Phi(\beta)=^\ast \underline{j})\wedge
\neg(\underline{i}=^\ast \underline{j}) ]$.
\end{center}
Either $i<j$ or $j<i$, and assume the first case. By proposition
\ref{Tproperty}(b), $\neg(A_\alpha\cap T_j\#\emptyset)$, i.e,
$A_\alpha\cap T_j=\emptyset$. So by $(a)$, $\neg(\delta_j\in
L_\alpha)$, and by the assumption $L_\alpha\cap E=L_\beta\cap E$,
we have $\neg(\delta_j\in L_\beta)$. By $(b)$, $\neg (A_\beta\cap
T_j\#\emptyset)$, that is, $A_\beta\cap T_j=\emptyset$. But it
contradicts with proposition \ref{Tproperty}(c), and thus
$\neg\neg(\Phi(\alpha)=^\ast \Phi(\beta))$. By assuming $\psi$, we
derived a contradiction. By
$(\psi\ra\varphi)\ra(\neg\neg\psi\ra\neg\neg\varphi)$, assuming
$\neg\neg\psi$ yield also a contraction. So
$\neg\neg(\Phi(\alpha)=^\ast \Phi(\beta))$.\end{proof}

The  theorem says that if $\Phi$ is a total constructive function
from the temporal interval $(0,1]^o$ to $\mathbb{N}^\ast$ then
there exists a \emph{non-infinite} subset $E$ of $[0,1]^o$ such
that
\begin{center} $\forall \alpha,\beta \in(0,1]^o [(L_\alpha\cap
E=L_\beta\cap E)\ra\neg\neg(\Phi(\alpha)=^\ast \Phi(\beta))]$.
\end{center}

  As  mentioned in the introduction, a moment $\alpha$ on the
temporal continuum is a \emph{flow} from the past, i.e., the
moment $\alpha$ is adhered to the collection of all moments that
happened before it.  We  think up of this flow as $L_\alpha$.
To construct a witness for a moment $\alpha$, the occurrence of
$\alpha$ is compared with the occurrence of \emph{non-infinitely}
many moments in $E$. If the results of the comparison are the same
for two moments $\alpha$ and $\beta$, then it is not false that
the witness constructed for $\alpha$ is also a witness for
$\beta$. In this way, the theorem~\ref{theorem-ocp} formalizes our
sense of the orientation discussed in the introduction.

We believe that the following form of theorem~\ref{theorem-ocp}
is plausible to be accepted as a principle, which we name the
\emph{oriented continuity principle}, $\OCP:$

\vspace{0.2cm}

\noindent \textbf{Oriented Continuity Principle:} \vspace{-0.1in}
\[
\begin{array}{l} $For~every~$ \Phi:(0,1]^o\ra
\mathbb{N}^\ast,  \qquad\qquad\qquad\qquad\qquad\qquad\qquad\qquad\qquad\qquad\qquad\qquad\\
$if$~\forall \alpha\in (0,1]^o~\exists \xi\in\mathbb{N}^\ast
~[\Phi(\alpha)=^\ast \xi]~
$then~there~exists~a~\emph{non-infinite}~subset$~   E\subseteq
[0,1]^o~$such~that$
 \\\forall \alpha,\beta\in (0,1]^o~
  [(L_\alpha\cap E=L_\beta\cap E)\ra (\Phi(\alpha)=^\ast\Phi(\beta))]. \\

\end{array} \]

\section{Topologies For Continuum}\label{topforc}

 What is the proper topology of the intuitive temporal continuum?
Clearly, the topology must display the notion of  orientation.
As is emphasized in the introduction, since $t$ sinks back to the
past, the moment $t$  is not distinguishable from moments in
$(10:40,11:30]$. Therefore, if we have a constructive method to
introduce witnesses  for all moments on the temporal line then
the witness for the moment $t$ is also the witness for all
moments in $(10 : 40, 11 : 30]$. Thus, for a suitable topology of
the temporal continuum, it seems that every total function is
continuous. 

 As Brouwer distinguishes between the intuitive  continuum and the
``full continuum" of the unfinished elements  of the unit segment
(intuitionistic real line) (\cite{kn:jk}, page 74), we do not
claim that our oriented line, equipped with topologies defined
below is exactly the intuitive temporal continuum based on our
intuition. We believe that the only difference between the
intuitive continuum and the intuitive temporal continuum is
taking  into account  the notion of  \emph{orientation} as a new
aspect of the continuum besides \emph{inexhaustibility} and  \emph{non-discreteness}. Our oriented line and the following
topologies are our suggestions for modeling  the temporal
continuum.

\subsection{The Oriented Topology}

\begin{definition}\textbf{The Oriented Topology on the temporal interval $(0,1]^o$.}
A subset  $U$ of $(0,1]^o$ is called open  if and only if for
every $\alpha$ in $U$ there exists a \emph{non-infinite} set
$E\subseteq [0,1]^o$ such that $S_E(\alpha)\subseteq U$, where
$S_E(\alpha)=\{\beta\in (0,1]^o\mid L_\alpha\cap E=L_\beta\cap
E\}$. We indicate the set of all open subsets  by
$\mathcal{T}_1$, and refer to $((0,1]^o,\mathcal{T}_1)$ as the
oriented topological space.

\end{definition}

\begin{example} {\em The interval $(\frac{1}{4},\frac{1}{2}]^o=\{\alpha\in (0,1]^o
\mid \hat{\frac{1}{4}} < \alpha \leq \hat{\frac{1}{2}}\}$  is
open. For $\alpha\in (\frac{1}{4},\frac{1}{2}]^o$, let
$E=\{\hat{\frac{1}{4}},\hat{\frac{1}{2}}\}$. Then $L_\alpha\cap
E=\{\hat{\frac{1}{4}}\}$, and $S_E(\alpha)\subseteq
(\frac{1}{4},\frac{1}{2}]$.}
\end{example}

We must show that the class of all open subsets of $(0,1]^o$ is a
topology. It is obvious that the empty set and $(0,1]^o$ belong to
$\mathcal{T}_1$.  One may easily see that the topology
$\mathcal{T}_1$ is closed under arbitrary union. The next
proposition shows that it is also closed under finite
intersection.

\begin{proposition}\label{topo}
The class of open subsets of $(0,1]^o$ is closed under finite
intersection.
\end{proposition}
\begin{proof}
Let $U_1$ and $U_2$ be open, and $\alpha\in U_1\cap U_2$. Let the
non-infinite sets $E_1,E_2$ be such that $S_{E_1}(\alpha)\subseteq
U_1$, and $S_{E_2}(\alpha)\subseteq U_2$. It is easily seen that
$S_{E_1\cup E_2}(\alpha)\subseteq U_1\cap U_2$.
\end{proof}

\subsection{The Ordinary Topology}
We first define the notion of semi-metric spaces, and then
introduce the \emph{ordinary} topology based on this notion.

\begin{definition}
Assume $\mathbb{Q}^+=\{q\in\mathbb{Q}\mid q>0\}$. A {\em
semi-metric space} is a couple $(I,\d)$, where $I$ is a set and \d
is a relation $\d\subseteq I\times I\times \mathbb{Q}^+$
satisfying the following properties
\begin{itemize}
\item[P1.] $\forall x\in I\forall q\in \mathbb{Q}^+ \d(x,x,q)$,

\item[P2.]$\forall x,y\in I \exists q\in \mathbb{Q}^+  \d(x,y,q)$,

\item[P3.]$\forall x,y\in I\forall q,p\in \mathbb{Q}^+
(\d(x,y,q)\wedge q<p\ra \d(x,y,p))$, \item[P4.] $\forall x,y\in
I\forall q\in \mathbb{Q}^+ (\d(x,y,q)\leftrightarrow \d(y,x,q))$,
\item[P5.] $\forall x,y,z\in I\forall q,p\in \mathbb{Q}^+
(\d(x,y,q)\wedge\d(y,z,p)\ra \d(x,z,p+q))$ \emph{(triangle
inequality)},
\end{itemize}
\end{definition}

The intended meaning of $\d(x,y,q)$ is `the distance of $x$ from
$y$ is less than $q$'. Then by this intention, all properties
P1-P5 are understood clearly. For $x\in I$ and $p\in
\mathbb{Q}^+$, let $S_p(x)=\{y\in I\mid \d(x,y,p)\}$.


\begin{remark}
{\em Note that two notions of metric space and semi-metric space
are {\em classically} the same. Assume $\mathbb{R}$ is the
classical real line. For a semi-metric space $(I,\d)$, we can
define a metric function $dist: I\times I\ra \mathbb{R}$, such that
$dist(x,y)=inf \{p\mid \d(x,y,p)\}$. One may easily verify that
$dist$ is a (classical) metric. Also if $(I,dist)$ is a
(classical) metric space, defining $\d\subseteq I\times I\times
\mathbb{Q}^{+}$ by $\d(x,y,q)\leftrightarrow dist(x,y)<q$, would
make $(I,\d)$ a semi-metric space.}
\end{remark}

\begin{proposition} If $(I,\d)$ is semi-metric space, then
the collection of subsets   $U\subseteq I$ satisfying the
following property is a topology:
\begin{center}
$\forall x\in U\exists p\in \mathbb{Q}^+ (S_p(x)\subseteq
U)~~(\ast)$.
\end{center}
\end{proposition}
\begin{proof}
It is clear that both sets $I$ and the empty set satisfy $(\ast)$.
We only need to show that the collection is closed under arbitrary
union and finite intersection.  Being closed under arbitrary union
is trivial. We prove that it is closed under finite intersection.
Assume $U_1$ and $U_2$ satisfy the property and let $x\in U_1\cap
U_2$. Since both   $U_1$ and $U_2$ satisfy the property, there
exist $p_1,p_2\in \mathbb{Q}^+$ such that $S_{p_1}(x)\subseteq
U_1$ and $S_{p_2}(x)\subseteq U_2$. Let $q=min(p_1,p_2)$. Then
$S_{q}(x)\subseteq U_1\cap U_2$. To show this, assume $y\in
S_{q}(x)$. Then $\d(x,y,q)$ holds, and by (P3), we have
$\d(x,y,p_1)$ and $\d(x,y,p_2)$, which implies $y\in
S_{p_1}(x)\cap S_{p_2}(x)\subseteq U_1\cap U_2$. Hence $U_1\cap
U_2$ satisfies the property.
\end{proof}

\begin{proposition} Let
$\d\subseteq \mathbb{R}^o\times  \mathbb{R}^o\times \mathbb{Q}^+$,
defined by \begin{center} $\d(\alpha,\beta,q):=\exists \zeta\in
\mathbf{Q}\exists p\in \mathbb{Q}^+ (p\leq q \wedge(
\hat{\zeta}\leq\alpha,\beta\leq\hat{\zeta}+p))$~\footnote{We use
expression $a\leq x,y\leq b$ as a short abbreviation of $a\leq
x\leq b\wedge a\leq y\leq b$.}.\end{center} Then
$(\mathbb{R}^o,\d)$ is a semi-metric space.
\end{proposition}
\begin{proof}
We must show that $\d$ satisfies properties P1-P5.
\begin{itemize}
\item P1. It follows from proposition \ref{app}.
\item P2. Let $\alpha,\beta$ be two oriented reals. There exist  rational
numbers $p_1,p_2,q_1,q_2$ such that $p_1 <\alpha< p_2 ,q_1<\beta <
q_2$. Let $\zeta\in \mathbf{Q}$ defined by $\zeta(n)=min(p_1,q_1)$
for each $n\in \mathbb{N}$. Then
\begin{center}
$\hat{\zeta} \leq\alpha,\beta\leq
(\hat{\zeta}+|min(p_1,q_1)|+max(p_2,q_2))$.
\end{center}
\item P3. Trivial.
\item P4. Trivial.
\item P5. Assume
$\alpha,\beta,\delta\in \mathbb{R}^o$ and for some $p,q\in
\mathbb{Q}$, $\d(\alpha,\beta,q)$ and $\d(\beta,\delta,p)$ hold.
Then there exist $\zeta_1,\zeta_2\in \mathbf{Q}$, $q_1\leq q$ and
$p_1\leq p$, such that
\begin{center}
$\hat{\zeta_1}\leq\alpha,\beta\leq(\hat{\zeta}_1+q_1)$
\end{center}
and
\begin{center}
$\hat{\zeta_2}\leq\beta,\delta\leq(\hat{\zeta_2}+p_1)$.
\end{center}
Let $\zeta(n)=min(\zeta_1(n),\zeta_2(n))$, for each $n$. We claim
that
\begin{center}
$\hat{\zeta}\leq\alpha,\delta\leq
(\hat{\zeta}+(p_1+q_1))$.
\end{center}
From $\hat{\zeta_1}\leq \alpha$ and $\hat{\zeta_2}\leq \alpha$,
we derive $\hat{\zeta}\leq \alpha$. We also derive
$\hat{\zeta}\leq \delta$ similarly. Assume an arbitrary
$t\in\mathbb{N}$. The fact $\alpha\leq \hat{\zeta_1}+q_1$ implies
that there exists $k\in \mathbb{N}$ such that
$\alpha(t)<\hat{\zeta_1}(k)+q_1$. On the other hand, since
$\hat{\zeta_1}\leq \beta$ and $\beta\leq \hat{\zeta_2}+p_1$,
 there exists $k'\in \mathbb{N}$ such that
$\hat{\zeta_1}(k)< \hat{\zeta_2}(k')+p_1$, and consequently
$\alpha(t)<\hat{\zeta_2}(k')+p_1+q_1$. Let $k''=max(k,k')$. Since
both of $\hat{\zeta_1},\hat{\zeta_2}$ are strictly increasing, we
have $\hat{\zeta_1}(k)< \hat{\zeta_1}(k''),~\hat{\zeta_2}(k')<
\hat{\zeta_2}(k'')$, and hence
$\alpha(t)<\hat{\zeta_1}(k'')+q_1+p_1$ and
$\alpha(t)<\hat{\zeta_2}(k'')+p_1+q_1$. Then
$\alpha(t)<(min(\zeta_1(k''),\zeta_2(k''))-\frac{1}{k''+1})+p_1+q_1$.
It is shown that $\alpha\leq \hat{\zeta}+p_1+q_1$. Similar
argument works for $\beta\leq \hat{\zeta}+p_1+q_1$.
\end{itemize}
\end{proof}

\begin{definition}\textbf{The Ordinary Topology on $\mathbb{R}^o$.}
The \emph{ordinary} topology on $\mathbb{R}^o$ is the topology
induced by the semi-metric $(\mathbb{R}^o,\d)$, where $\d$ is
defined above. We show the \emph{ordinary} topological space by
$(\mathbb{R}^o,\mathcal{T}_2)$
\end{definition}


\subsection{A consequence  of OCP }
We use \OCP to prove that:
\begin{theorem}\label{totalc}Every total function from
$((0,1]^o, \mathcal{T}_1)$ to $(\mathbb{R}^o, \mathcal{T}_2)$ is
continuous.
\end{theorem}
\begin{proof} Let $f: ((0,1]^o, \mathcal{T}_1)\ra
(\mathbb{R}^o, \mathcal{T}_2)$ be total. We prove that for each
$n\in \mathbb{N}$, for every $\alpha\in(0,1]^o$, there exists
$S_{E_n}(\alpha)\in \mathcal{T}_1$ such that  for every $\beta\in
S_{E_n}(\alpha)$, $\d(f(\alpha),f(\beta),2^{-n})$.
Consider rational numbers $q_i=i2^{-n}$, $0\leq i\leq 2^{n}$. For
each $\delta\in(0,1]^o$, define $\phi(f(\delta))(0)=0$ and
$\phi(f(\delta))(k)=i$ if and only if $q_i\leq
f(\delta)(k)<q_{i+1}$. The  function
$\Phi(\delta)=\varphi(f(\delta))$ from $(0,1]^o$ to
$\mathbb{N}^\ast$ is total and well-defined. By \OCP, there exists
a non-infinite subset $E\subseteq [0,1]^o$, such that for every
$\alpha,\beta$, if $\beta\in S_E(\alpha)$ then
$\Phi(\alpha)=^\ast\Phi(\beta)$. It easily seen that for
$\zeta_1(i)=q_{\Phi(\alpha)(i)}\in \mathbf{Q}$ and
$\zeta_2(i)=q_{\Phi(\beta)(i)}\in \mathbf{Q}$, we have
$\zeta_1=\zeta_2$ and
\begin{center}
$\hat{\zeta_1}\leq
f(\alpha),f(\beta)\leq(\hat{\zeta}_1+2^{-n})$,
\end{center} i.e.,
$\d(f(\alpha),f(\beta),2^{-n})$.
\end{proof}


\section{Concluding Remarks}

  Our temporal continuum suggests a new framework for constructive
analysis.  We  need to investigate  the following questions as further work:
\begin{enumerate}

\item Is every total function on $\mathbb{R}^o$ with {\em one}
topology (either the oriented topology or the ordinary one)
continuous, or do we need a third topology?
\item What happens to the \emph{uniform
continuity theorem}?
\item What does the \emph{intermediate value theorem} looks like?
\end{enumerate}
\vspace{0.5cm}

\textbf{Acknowledgement}. We are very thankful to Dirk van Dalen,
Mark van Atten and Wim Veldman for reading a draft of this paper
and for their  helpful comments and suggestions. Moreover, our
correspondence with Wim Veldman resulted in   modifications of
some crucial claims in this paper.

\end{document}